\documentclass[11pt]{article}
\usepackage{amssymb}
\usepackage{amsthm}
\usepackage{amsmath}
\usepackage{hyperref}
\pdfoutput=1
\newtheorem{theorem}{Theorem}[section]

\newtheorem{proposition}[theorem]{Proposition}

\begin{document}
\title{$\omega$-Euclidean domain and skew Laurent series rings}
\author{Oleh Romaniv and Andrij Sagan \\
\\\vspace{6pt} Department of Mechanics and Mathematics,
\\ Ivan Franko National University of Lviv, Ukraine}

\maketitle

\begin{abstract}
In this paper we proved that if $R$ is right $\omega$-Euclidean domain, then skew Laurent formal series ring is right $\omega$-Euclidean domain. We also showed that if $R$ is a right $\omega$-Euclidean domain with multiplicative norm, then skew Laurent formal series ring is a right principal ideal domain. In addition, we proved that if $R$ is a noncommutative $\omega$-Euclidean domain with a multiplicative norm, then $R$ and skew Laurent formal series ring is a ring with elementary reduction of matrices.
\end{abstract}

% main text

\section{Introduction}

In the present paper, $R$ is understood as an associative domain with nonzero unit element, $\sigma$ is understood as an automorphism of $R$ and $\mathcal{U}(R)$ is understood as the group of invertible elements of a ring $R$. Let $\varphi: R\rightarrow\mathbb{N}\cup\{0\}$ be a function satisfying the following condition: $\varphi(a)=0$ if and only if $a=0$; $ \varphi(a)>0$ for any nonzero and $\varphi(ab)\geq \varphi(a)$ for any arbitrary elements $a,b\in R$. This function is called the \textit{norm} over domain $R$.

Domain $R$ is called \textit{a right Euclidean} if for any arbitrary elements $a,b\in R$ with $b\neq 0$, there exist $q,r\in R$ such that $a=bq+r$ and $\varphi(r)<\varphi(b)$.

A \textit{right $k$-stage division chain} \cite{Zab} for any arbitrary elements $a,b\in R$ with $b\neq 0$ is understood as the sequence of equalities
\begin{equation}\label{2.1}
  a=bq_{1}+r_{1}, \\
  b=r_{1}q_{2}+r_{2},\\
  \ldots\ldots\ldots, \\
  r_{k-2}=r_{k-1}q_{k}+r_{k},
\end{equation}
with $k\in \mathbb{N}$.

Domain $R$ is called \textit{a right $\omega$-Euclidean domain}\cite{Zab} with respect to the norm $\varphi$, if for any arbitrary elements $a,b\in R$, $b\neq 0$, there exists a right $k$-stage division chain $\eqref{2.1}$ for some $k$, such as $\varphi(r_{k})<\varphi(b)$.

Clearly, the right Euclidean domain is a right $\omega$-Euclidean domain.

A ring $R$ is called \textit{a ring with elementary reduction of
matrices}\cite{Zab3} in case when an arbitrary matrix over $R$
possesses elementary reduction, i.e. for an arbitrary matrix $A$
over the ring $R$  there exist such elementary matrices over $R$ ,$\;
P_{1},\ldots, P_{k}, Q_{1},\ldots, Q_{s}$ of respectful size such that
\begin{equation*}\label{eq_1}
P_{1}\cdots P_{k}\cdot A\cdot Q_{1}\cdots Q_{s} = diag(\varepsilon_{1},\ldots, \varepsilon_{r}, 0,\ldots, 0),
\end{equation*}
where $R\varepsilon_{i+1}R\subseteq R\varepsilon_{i}\cap \varepsilon_{i}R$ for any $i= 1,\ldots, r-1$.

Let $R_{X}=R[[x,x^{-1};\sigma]]$ be the \textit{skew Laurent power series ring} in which each element has a unique representation as
$$
\sum\limits_{i=h}^{\infty}a_{i}x^{i},\quad h\in\mathbb{Z},\quad a_{i}\in R
$$
with $a_{-n}$ for almost all positive integer $n$ and $x^{i}\cdot a=\sigma^{i}(a)\cdot x^{i}$.

P. Samuel in \cite{Samuel} proved that if $R$ is a Euclidean domain then Laurent formal series ring is a Euclidean domain. Also K. Amano and T. Kuzumaki in \cite{Amano} proved that if $R$ is a right Euclidean domain then a skew Laurent formal series ring is a Euclidean domain. O. M. Romaniv and A. V. Sagan in \cite{RomSag5} proved that $R$ is a commutative $\omega$-Euclidean domain if and only if a Laurent formal series ring is a $\omega$-Euclidean domain. In this paper, the results  in \cite{Samuel} and \cite{RomSag5} for a skew Laurent formal series ring are generalized. In addition, elementary reduction of matrices over a skew Laurent formal series ring are investigated.

Note that if $\sigma = 1_{R}$ then a skew Laurent formal series ring is a Laurent formal series ring.

\section*{Main results}

Let $R$ be an integral domain with a norm map $\varphi: R\rightarrow \mathbb{N}\cup\{0\}$. For any arbitrary element
$$
f=\sum\limits_{i\geq h}a_{i}x^{i}\in R_{X},\quad a_{i}\in R,\; a_{h}\neq 0
$$
we put $\psi(f)=a_{h}$ and $\psi(f)=0$ if and only if $f=0$. So $\psi$ is a map from $R_{X}$ to $R$.

For any arbitrary elements $a, b\in R$ a record $b \mid a$ means that $a \equiv 0\;(mod\; b)$ otherwise $b\nmid a$.

\begin{proposition}\label{lemma1}
For any arbitrary elements $f, g\in R_{X}$ with $g\neq 0$ we have that $f=gu$ or $f=gu+v$, where $\psi(g)\nmid \psi(v)$.
\end{proposition}

\begin{proof}
Let $h$ be the lowest degree of $f$ and $k$ be the lowest degree of $g$. Set $\psi(f)=\psi(g)q+r$ where $q, r\in R$.

Since $\sigma$ is an automorphism of $R$, then there exists some element $c\in R$, such that $\sigma^{k}(c)=q$. In this case, we write $v = f - gcx^{h-k}$. Then
\begin{eqnarray*}
v = (\psi(f) - \psi(g)q)x^{h}\; +\; \text{higher degree terms} = \\
= rx^{h}\; +\; \text{higher degree terms},
\end{eqnarray*}
where $r\in R$.

If $\psi(g)\nmid r$, then we get the desired result $\psi(g)\nmid r = \psi(v)$ and stop the process.

If $\psi(g)\mid r$, then $\psi(v) = \psi(g)q_{1}$ and $\sigma^{k}(c_{1}) = q_{1}$ for some $c_{1}\in R$. Write
$$
v_{1}=v-gc_{1}x^{l_{1}-k},\; \text{where}\; l_{1}\; \text{is the lowest degree of}\; v.
$$
Then, we again consider two possible cases $\psi(g)\nmid \psi(v_{1})$ or $\psi(g)\mid \psi(v_{1})$ and so on. If this process is finite, then
$$
f = g(cx^{h-k}+c_{1}x^{l_{1}-k}+\cdots + c_{n}x^{l_{n}-k}) + v_{n},
$$
where $\psi(g)\nmid \psi(v_{n})$.

If this process continues infinite number of steps, we get
$$
u=сx^{h-k}+с_{1}x^{l_{1}-k}+\cdots + c_{n}x^{l_{n}-k} + \cdots,
$$
so $f=gu$.
\end{proof}

We define a map $\varphi_{x}: R_{X}\rightarrow\mathbb{N}\cup\{0\}$ by $\varphi_{x}(f)=\varphi(\psi(f))$.

Then we obtain the following result.

\begin{theorem}\label{th 1}
Let $R$ be a right $\omega$-Euclidean domain with respect to $\varphi$, then $R_{X}$ is a right $\omega$-Euclidean domain with respect to $\varphi_{x}$.
\end{theorem}

\begin{proof}
Consider any arbitrary two elements $f, g\in R_{X}$ where $g\neq 0$,
\begin{gather*}
f = \psi(f)x^{h} + \text{higher degree terms,} \\
g=\psi(g)x^{k} + \text{higher degree terms.}
\end{gather*}
Then by Proposition \ref{lemma1} we have either following equality $f=gu$ or another equality $f=gu+v$ where $\psi(g)\nmid \psi(v)$.

Obviously, in the first case, $R_{X}$ is a right Euclidean domain, and, therefore, a right $\omega$-Euclidean domain.

In the second case, consider two possible cases delete under:

\item[(1)] $\varphi_{x}(v)<\varphi_{x}(g)$. By the definition, $R_{X}$ is a right Euclidean domain and, hence, $R_{X}$ is a right $\omega$-Euclidean domain.

\item[(2)] $\varphi_{x}(v)\geqslant\varphi_{x}(g)$. Then there exists a right $k$-stage division chain
$$
\psi(v) = \psi(g)q_{1}+r_{1},\; \psi(g) = r_{1}q_{2}+r_{2},\; \ldots, \; r_{k-2} = r_{k-1}q_{k} + r_{k},
$$
where $\varphi(r_{k})<\varphi(\psi(g))$ (because $R$ is a right $\omega$-Euclidean domain).

Then there exists some element $c_{1}\in R$, such that $\sigma^{k}(c_{1}) = q_{1}$, where $q_{1}\in R$. Now we can write
$$
v - gc_{1}x^{l_{1}-k} = v_{1} \quad \text{(where $l_{1}$ --- order $v$)},
$$
hence we have $f = g(u - с_{1}x^{l_{1}-k}) + v_{1}$ and $\psi(v_{1}) = r_{1}$.

Continuing this process, we find that there exists some element $c_{2}\in R$ such that $\sigma^{l_{2}}(c_{2}) = q_{2}$, where $q_{2}\in R$.
Write
$$
g - v_{1}c_{2}x^{k-l_{2}} = v_{2} \quad \text{(where $l_{2}$ --- order $v_{1}$)},
$$
hence we have $g = v_{1}c_{2}x^{k-l_{2}} + v_{2}$ and $\psi(v_{2}) = r_{2}$.

This process will continue until to some $k\in \mathbb{N}$: then there exists some $c_{k}\in R$, such that $\sigma^{l_{k}}(c_{k}) = q_{k}$, where $q_{k}\in R$, and write
$$
v_{k-2} - v_{k-1}c_{k}x^{l_{k-1}-l_{k}} = v_{k} \quad \text{(where $l_{k}$ --- order $v_{k-1}$)}.
$$
From here $v_{k-2} = v_{k-1}c_{k}x^{l_{k-1}-l_{k}} + v_{k}$ and $\psi(v_{k}) = r_{k}$.

Summing up the above, we have a right $k$-stage division chain
\begin{equation}\label{3.1}
f = g(u - q_{1}x^{l_{1}-k}) + v_{1}, g = v_{1}c_{2}x^{k-l_{2}} + v_{2}, \ldots, v_{k-2} = v_{k-1}c_{k}x^{l_{k-1}-l_{k}} + v_{k}.
\end{equation}

If $r_{k}\neq 0$, then we have
$$
\varphi_{x}(g) = \varphi(\psi(g))> \varphi(r_{k}) = \varphi_{x}(v_{k}).
$$
and, therefore $R_{X}$ is a right $\omega$-Euclidean domain.

If $r_{k} = 0$, then obviously
$$
\varphi_{x}(g) = \varphi(\psi(g))> \varphi(r_{k}) = \varphi_{x}(v_{k}) = 0.
$$
Note that by the properties of the norm we have $v_{k}=0$ and, therefore, the chain \eqref{3.1} is finite, and, therefore, by Proposition 1 \cite{Cooke}, $R_{X}$ is a right $\omega$-Euclidean domain.
\end{proof}

Clearly, for a left $\omega$-Euclidean domains, the left-side version of Proposition\ref{lemma1} is also valid.

\begin{theorem}\label{th2}
Let $R$ be a left principal ideal domain, a right $\omega$-Euclidean domain. Then
\begin{enumerate}
  \item[(1)] $R$ is a ring with elementary reduction of matrices;
  \item[(2)] $R_{X}$ is a ring with elementary reduction of matrices.
\end{enumerate}
\end{theorem}

\begin{proof}
(1) Let $R$ be a left principal ideal domain and a right $\omega$-Euclidean domain. Then by Proposition 4 \cite{Cooke} and Theorem 3.6 \cite{Cohn1111} it follows that $R$ is a ring with elementary reduction of matrices.

(2) Let $R$ be a left principal ideal domain and a right $\omega$-Euclidean domain, then by Theorem 6.3 \cite{Tugan} and Theorem \ref{th 1}, we have that $R_{X}$ is a left principal ideal domain and a right $\omega$-Euclidean domain. Then by the subsection (1), $R_{X}$ is a ring with elementary reduction of matrices.
\end{proof}
 The norm $\varphi$ over $R$ is multiplicative, if for any arbitrary elements $a, b\in R$ we have
$$
\varphi(a\cdot b) = \varphi(a)\cdot \varphi(b).
$$
As a result, we find that $\varphi(u) =1$ for any arbitrary element $u \in \mathcal{U}(R)$.

\begin{theorem}\label{th3}
Let $R$ be a right $\omega$-Euclidean domain with respect a multiplicative norm to $\varphi$. Then
\begin{enumerate}
  \item[(1)] $R$ is a right principal ideal domain;
  \item[(2)] $R_{X}$ is a right principal ideal domain.
\end{enumerate}
\end{theorem}

\begin{proof}
(1) Let $I$ be a right ideal of $R$. The case of $I = \{0\}$ is trivial. Therefore, let $I$ be at least one nonzero element. Among all nonzero elements $I$ let us choose the element with the lowest norm. Suppose that this element is $b\in I$. Now consider an arbitrary nonzero element $a\in I$. Because $R$ is a right $\omega$-Euclidean domain, then for the elements $a$ and $b$ there exists a right $k$-stage division chain:
$$
 a=bq_{1}+r_{1}, \; b=r_{1}q_{2}+r_{2},\;\cdots,\; r_{k-2}=r_{k-1}q_{k}+r_{k},
$$
where $\varphi(r_{k})<\varphi(b)$. It is clear that elements $r_{1},r_{2},\ldots,r_{k}\in I$.

By the construction, the element $b$ has the lowest norm, because the inequality \\$\varphi(r_{k})<\varphi(b)$ is possible only when $r_{k} = 0$. Hence, the elements $a$ and $b$ have the greatest common divisor element $r_{k-1}$, therefore $a = r_{k-1}a_{0}$ and $b = r_{k-1}b_{0}$.

However, since the norm $\varphi$ is multiplicative, delete therefore $\varphi(b) = \varphi(r_{k-1})\cdot \varphi(b_{0})$.

Since $b$ has the lowest norm, it follows that $\varphi(r_{k-1}) = 1$ or $\varphi(b_{0}) = 1$. In this case, when $\varphi(r_{k-1}) = 1$ is trivial. Therefore, $\varphi(b_{0}) = 1$. Hence, $b_{0}\in \mathcal{U}(R)$. Then $r_{k-1} = bb_{0}^{-1}$ and, therefore, we find that $a = bb_{0}^{-1}a_{0}$. So $I\subset bR$. On the other hand, the element $b\in I$. Therefore, $I\supset aR$. Two previous inclusions shows that $I = aR$. Hence, $R$ is a right principal ideal domain.

(2) The proof follows from Theorem \ref{th 1} and the first half of the proof, thus, we only need to show that if $\varphi$ is the multiplicative norm, then $\varphi_{x}$ is the multiplicative norm. Hence, $\psi(fg) = \psi(f)\sigma^{h}(\psi(g))$, where $h$ is of the order $f$.
Then
\begin{eqnarray*}
\varphi_{x}(fg) = \varphi(\psi(fg)) = \varphi(\psi(f)\sigma^{h}(\psi(g))) = \varphi(\psi(f))\varphi(\sigma^{h}(\psi(g))) =\\
= \varphi(\psi(f))\varphi(\psi(g)) = \varphi_{x}(f)\cdot\varphi_{x}(g).
\end{eqnarray*}
Hence, $\varphi_{x}$ is the multiplicative norm. Note that $\varphi(a) = \varphi(\sigma(a))$. This completes the proof.

\end{proof}

\begin{theorem}\label{th4}
Let $R$ be a noncommutative $\omega$-Euclidean domain with the multiplicative norm delete to $\varphi$. Then
\begin{enumerate}
  \item[(1)] $R$ is a ring with elementary reduction of matrices;
  \item[(2)] $R_{X}$ is a ring with elementary reduction of matrices.
\end{enumerate}
\end{theorem}

\begin{proof}
By Theorem \ref{th3}, $R$ and $R_{X}$ are right principal ideal domains, then by Theorem \ref{th2} we obtain the statement of the theorem.
\end{proof}


\begin{thebibliography}{20}

\bibitem{Amano}
K.~Amano, T.~Kuzumaki \emph{On Euclidean algorithm and skew Laurent power series rings.} Bull. Fac. Gen. Ed. Gifu Univ. \textbf{38} 1995. pp. 233--237.

\bibitem{Cohn1111}
P.~M.~Cohn \emph{Right principal Bezout domain.} J. London Math. Soc. \textbf{35(2)} 1987. pp. 251--262.

\bibitem{Cooke}
G.~E.~Cooke \emph{A weakening of the euclidean property for integral domains and applications to algebraic number theory. I.} Journal fur die Reine und angewande Math. \textbf{282} 1976. pp. 133--156.

\bibitem{RomSag5}
O.~M.~Romaniv, A.~V.~Sagan \emph{$\omega$-euclidean domain and Laurent series.} Carpathian Math. Publ. \textbf{8(1)} 2016. pp. 158--162.

\bibitem{Samuel}
P.~Samuel \emph{About Euclidean Rings.} J. Algebra  \textbf{19} 1971. pp. 282--301.

\bibitem{Tugan}
D.~A.~Tuganbaev \emph{Laurent rings.} J. Math. Sci.  \textbf{149(3)} 2008. pp. 1286--1337.

\bibitem{Zab3}
B.~V.~Zabavsky \emph{Rings with elementary reduction matrix.} Ring Theory Conf., Miskolc, July 15-20, (1996) pp.14.

\bibitem{Zab}
B.~V.~Zabavsky, O.~M.~Romaniv \emph{Noncommutative rings with elementary reduction of matrices.} Visnyk of the Lviv Univ. Series Mech. Math. \textbf{49} 1998. pp. 16--20.

\end{thebibliography}
\end{document}